\documentclass[11pt,a4paper]{amsart}%
\usepackage{amssymb}
\usepackage{epsfig}
\usepackage{graphicx}
\usepackage{algorithm}
\usepackage{algorithmic}
\usepackage{pstricks,pst-node,pst-plot}
\usepackage[centertags]{amsmath}
\usepackage{amsfonts}
\usepackage{amsthm}
\usepackage{newlfont}
\usepackage{subfigure}
\usepackage{graphicx}
\usepackage{tikz}

\theoremstyle{plain}
\newtheorem{theorem}{Theorem}[section]
\newtheorem{conjecture}[theorem]{Conjecture}
\newtheorem{corollary}[theorem]{Corollary}
\newtheorem{lemma}[theorem]{Lemma}

\newtheorem{example}[theorem]{Example}

\theoremstyle{remark}

\theoremstyle{definition}

\begin{document}
\title[A survey on graphs with convex quadratic stability number]{\sc A survey on graphs with convex quadratic stability number}
\author{Domingos M. Cardoso}
\address{Center for Research and Development in Mathematics and Applications,
Department of Mathematics, University of Aveiro, 3810-193 Aveiro, Portugal}
\email{dcardoso@ua.pt}
\date{\today }
\date{}
\date{\today }
\date{}
\maketitle

\begin{abstract}
A graph with convex quadratic stability number is a graph for which the stability number is determined by solving a convex quadratic program.
Since the very beginning, where a convex quadratic programming upper bound on the stability number was introduced, a necessary and sufficient
condition for this upper bound be attained was deduced.  The recognition of graphs  with convex quadratic stability number has been deeply
studied with several consequences from continuous and combinatorial point of view. This survey starts with an exposition of
some extensions of the classical Motzkin-Straus approach to the determination of the stability number of a graph and its relations with the
convex quadratic programming upper bound. The main advances, including several properties and alternative characterizations of graphs with
convex quadratic stability number are described as well as the algorithmic strategies developed for their recognition. Open problems and a
conjecture for a particular class of graphs, herein called adverse graphs, are presented, pointing out a research line which is a challenge
between continuous and discrete problems.
\end{abstract}

\section{Introduction}
The use of quadratic programming as a model for determining the stability number of a graph dates back to 1965, with the publication of Motzkin
and Straus \cite{MotzkinStraus65}. In this publication, the clique number of a graph (and then the stability number of its complement) is directly
determined from the optimal value of a quadratic programming problem with a very simple formulation which uses the adjacency matrix of the graph.
Despite this approach has no implication on the complexity of the determination of the clique (stability) number of a graph, the obtained relation
between a combinatorial parameter and the optimal value of a continuous optimization problem was very surprising and became a strong motivation
for the research on the application of continuous optimization to solve combinatorial problems. The quadratic programming approach to the
clique (stability) number of a graph has been deeply studied in several papers, namely in \cite{Bomze, Pelilo95, PeliloJagota95}. In \cite{Luz95}
a convex quadratic programming upper bound on the stability number of a graph was introduced and a necessary and sufficient condition for
this upper bound be attained was proved. This approach was extended to graphs with vertex weights in \cite{LuzCardoso2}. When the graph is regular,
as it was firstly proved in \cite{Luz95}, the optimal value of this convex quadratic program coincides with the very popular upper bound on the
stability number of regular graphs obtained by Hoffman (unpblished) and presented by Lov\'asz in \cite{Lovasz79}.
This type of convex quadratic programming approach was related with the Lov\'asz theta number of a graph in \cite{LuzSchrijver05} and \cite{Luz16}.
The graphs for which the convex quadratic upper bound is attained by the stability number were called in \cite{Cardoso01},
graphs with convex quadratic stability number or simply graphs with convex-$QP$ stability number, where $QP$ means quadratic programming.
The search for the recognition of graphs with convex-$QP$ stability number has produced several papers with several improvements on
the way of finding a polynomial-time algorithm for such recognition
\cite{Cardoso01, Cardoso03, CardosoCvetkovic06, CardosoLuz16, CaviqueLuz09, Luz95-2, LuzCardoso1}. Notice that, for every graph $G$, we may determine
in polynomial-time the optimal value of a convex quadratic programming problem associated to $G$ (see \eqref{Luz_UpperBound}, in Section~\ref{Convex-QP}),
and one of the following two cases can occur: either such optimal value coincides with the stability number $\alpha(G)$ or it is greater than $\alpha(G)$,
otherwise. Despite all the achievements, in general, we are not able to know in which of those two cases is the graph $G$, that is, deciding whether
$G$ has or not convex-$QP$ stability number remains as an open problem. So far, the nature of the approaches to attack this problem is either continuous
or combinatory or sometimes a mixture of both, making this line of research a challenge between continuous and discreet problems, as it can be seen
throughout the text.\\

This is an expository article surveying the main published results on graphs with convex-$QP$ stability number, including the proofs of some results.
The advances on the direction of the polynomial-time recognition of graphs with convex-$QP$ stability number are described and the open problems are
presented. \\

The remaining part of this section is devoted to the notation and basic definitions.
In Section~\ref{MS_Extensions} the relations between the classical Motzkin-Straus approach and its extensions with the convex quadratic
programming model $(P(G))$ introduced in \cite{Luz95} are analyzed. The main focus of Section~\ref{Convex-QP} is the characterization of
graphs with convex quadratic stability number. Furthermore, the properties of these graphs are analyzed, leading to the construction of
bridges between continuous and discrete optimization. In Section~\ref{Recognition}, several results related with the recognition of graphs
with convex quadratic stability number are presented, namely the main advances on the direction of an algorithmic strategy with which one
can attain the polynomial-time recognition of these graphs.
So far, such polynomial-time recognition has resisted to be completed solved and it is currently an interesting open question.
Despite this, two computational effective algorithmic approaches were considered: the one based on the determination of $(\kappa,\tau)$-regular
sets and the simplex-like approach based in the concepts of star set/star complement of the least eigenvalue of a graph. With such algorithms
we may recognize in polynomial-time if a graph has or not convex-$QP$ stability number when it belongs to particular families, as it is the
case of bipartite graphs among some other families \cite{CardosoLuz01, Cardoso03}. In fact, the polynomial-time recognition of graphs with
convex-$QP$ stability number is available for all the graphs with exception of graphs with a so called adverse subgraph for which there is a
conjecture.\\

We deal with undirected simple graphs $G=(V(G),E(G))$, where $V(G)$ denotes the nonempty set of vertices and $E(G)$ the set of edges.
It is also assumed that $G$ is of order $n \ge 1,$ that is, $|V(G)|=n \ge 1$ and its size is $|E(G)|$. An element of  $E(G)$, with end vertices $i$
and $j$, is denoted by $ij$ (or $ji$) and we say that the vertex $i$ is adjacent to the vertex $j.$ The neighborhood of a vertex $v \in V(G)$ is
$N_{G}(v) = \{w: vw \in E(G)\}$ and its degree is $d_{G}(v)=|N_{G}(v)|$. A graph of order $n$ in which all pairs of
vertices are adjacent is a complete graph $K_{n}.$ A bipartite graph $H$ is a graph such that its vertex set $V(H)$ can be partitioned
into non-empty vertex subsets $V_{1}$ and $V_{2}$ and every edge has one end vertex in $V_1$ and the other in $V_2$.
When $d_G(i)=|V_2|$ for all $i \in V_1$ (and then $d_G(j)=|V_1|$ for all $j \in V_2$) the bipartite graph is called complete
bipartite and it is denoted by $K_{pq}$, where $p=|V_1|$ and $q=|V_2|$. The graph $G$ is $p$-regular if $d_{G}(v) = p$ for all $v \in V(G)$.
Given a vertex subset $U \subset V(G),$ the subgraph of $G$ induced by $U,$ $G[U],$ is such that $V(G[U])=U$ and
$E(G[U])=\{ij: i,j \in U \wedge ij \in E(G)\}$ and $G-U$ is the graph obtained from $G$ after deleting all the vertices in $U$, that is,
if $T=V(G) \setminus U$, then $G-U=G[T]$.  The complement of a graph $G$ is the graph $\overline{G}$ such that $V(\overline{G})=V(G)$ and
$E(\overline{G}) = \{ij: i,j \in V(G) \wedge ij \not \in E(G)\}.$
The line graph of the graph $G,$ $L(G)$, is constructed by taking the edges of $G$ as vertices of $L(G)$ and joining two vertices in $L(G)$
by an edge, whenever the corresponding edges in $G$ have a common vertex. A vertex set $S$ is called a stable (clique) set if no (every) two
vertices are adjacent. A stable (clique) set $S$ is called maximum stable (clique) set if there is no other stable (clique) set with greater
number of vertices. The number of vertices in a maximum stable (clique) set of a graph $G$, is called the stability (clique) number of $G$
and it is denoted by $\alpha(G)$ ($\omega(G)$). It is immediate that $\alpha(G) = \omega(\overline{G})$ and then the determination of the
stability number is equivalent to the determination of the clique number. As proved in \cite{Karp72}, given a nonnegative integer $k$, to
determine if a graph $G$ has a stable set of cardinality $k$ is $NP$-complete and therefore, in general, the determination of $\alpha(G)$
(as well as $\omega(G)$) is a hard problem. However, there are several particular classes of graphs for which the stability number can be
determined  in polynomial-time as it is the case of perfect graphs \cite{Lovasz86}, claw-free graphs \cite{Minty80} and \cite{Sbihi80},
$(P_5,banner)$-free graphs \cite{Lozin00}, among many others. A matching in a graph $G$ is a subset of edges, $M \subseteq E(G),$ no two of
which have a common  vertex. A matching with maximum cardinality is designated maximum matching. Furthermore, if for each vertex $v \in V(G)$
there is one edge of the matching $M$ incident with $v,$ then $M$ is called a perfect matching. Notice that the determination of a maximum
stable set of a line graph $L(G)$ is equivalent to the determination of a maximum matching of $G.$ Based on the Edmonds maximum matching
algorithm \cite{Edmonds65}, there are several polynomial-time algorithms for the determination of a maximum matching of a graph. Given a
vertex subset $S$ of a graph $G,$ the vector $x \in \mathbb{R}^{V}$ with $x_{v}=1$ if $v \in S$ and $x_{v}=0$ if $v \notin S$ is called the
characteristic vector of $S.$  The adjacency matrix of a graph $G$ of order $n$ is the symmetric matrix $A_{G}=\left( a_{ij}\right)_{n\times n}$
such that
$$
a_{ij}=\left\{\begin{array}{cl}
                     1 & ,\mbox{if } ij \in E(G) \\
                     0 & ,\mbox{otherwise.}
              \end{array}
       \right.
$$
From the symmetry of the matrix $A_G$, it follows that all its $n$ eigenvalues, which will be denoted by $\lambda_1(G), \lambda_2(G), \dots, \lambda_n(G)$, are real.
As usually, throughout this text, they will be consider (with possible repetitions) in non increasing order, that is, $\lambda_1(G) \ge \lambda_2(G) \ge \dots \ge \lambda_n(G)$.
The least eigenvalue of $A_G$, $\lambda_n(G)$, it will be also denoted by $\lambda_{min}(G)$. The eigenvalues of $A_G$ are also called the eigenvalues of
$G$. It is worth mention that if a graph $G$ has at least one edge, then $\lambda_{min}(G) \leq -1.$ Actually, $\lambda_{min}(G)=0$ iff $G$ has no edges,
otherwise $\lambda_{min}(G)=-1$ iff every component is complete. In the other cases, $\lambda_{min}(G) \le -\sqrt{2}$ \cite{Doob82}.\\

For the remaining notation and concepts we refer the books \cite{BrouwerHaemers12} or \cite{CvetRowSim2010}.

\section{Extensions of the Motzkin-Straus approach to the determination of the clique number of graphs}\label{MS_Extensions}

Let us consider a graph $G$ of order $n$ and the family of quadratic programming problems
\begin{eqnarray}
\left(P_{G}(\tau)\right) & \qquad & \upsilon_{G}\left(\tau\right) = \max \{2\hat{e}^{T}x-x^{T}\left( \frac{1}{\tau}A_{G}+I_n \right)x: x \ge 0 \}, \label{quadratic_parametric_family}
\end{eqnarray}
with $\tau>0$, where $\hat{e}$ is the all one vector and $I_{n}$ is the identity matrix of order $n$.\\

The following result can be obtained using a similar proof to the one obtained in \cite[Th. 2]{Cardoso01}.

\begin{theorem}\label{optimality_condition_1}
If $x^{*}(\tau)$ is an optimal solution for $(P_{G}(\tau))$, then
\begin{equation}
\forall i \in V(G) \qquad [x^{*}(\tau)]_{i} = \max\{0,1 - \frac{aî_Gx^*}{\tau}\},\label{discrete optimality_condition}
\end{equation}
where $[x^{*}(\tau)]_{i}$ is the $i$-the component of $x^{*}(\tau)$ and $aî_G$ is the $i$-th row of the matrix $A_G$.
\end{theorem}

When $\tau \ge - \lambda_{min}(G)$, the objective function of $(P_G(\tau))$ is convex and then the condition \eqref{discrete optimality_condition} is also
sufficient for the optimality of $x^{*}(\tau)$ \cite{Cardoso01}.

As a consequence of Theorem~\ref{optimality_condition_1}, it follows that $0 \le [x^{*}(\tau)]_{i} \leq 1 \;\; \forall i \in V(G)$ and thus
$
0 \le x^{*}(\tau) \le \hat{e} \; \Rightarrow \; 2\hat{e}^Tx^{*}(\tau) - {x^{*}(\tau)}^T\left(\frac{A_G}{\tau} + I_n\right)x^{*}(\tau) \le 2\hat{e}^Tx^{*}(\tau) - \|x^{*}(\tau)\|^2 \le n.
$
On the other hand, since $2\hat{e}^T\hat{e}_j - \hat{e}_j^T\left(\frac{A_G}{\tau} + I_n\right)\hat{e}_j=1$, where $\hat{e}_j$ denotes the $j$-th
vector of the canonical basis of $\mathbb{R}^n$, it follow that
\begin{equation}
\forall \tau>0 \qquad 1 \leq \upsilon_{G}(\tau) \leq n, \label{upsilon_bounds}
\end{equation}
with $\upsilon_{G}(\tau)=1$ if $G$ is a clique and $\upsilon_{G}(\tau)=n$ if $G$ has no edges.

\begin{theorem}\cite{CardosoLuz01}\label{cardoso_luz_th}
Consider a graph $G$ of order $n$ and the function $\upsilon_{G} : ]0, +\infty[ \; \mapsto  \; [1, n]$ such that $\upsilon(\tau)$ is defined in
\eqref{quadratic_parametric_family}. Then
\begin{enumerate}
\item $\forall \tau > 0 \;\; \alpha(G) \le \upsilon_{G}(\tau)$;\label{CL1}
\item $0 < \tau_{1} < \tau_{2} \; \Rightarrow \; \upsilon_{G}(\tau_{1}) \le \upsilon_{G}(\tau_{2})$;\label{CL2}
\item Assuming $\tau^{*}>0$ the following statements are equivalent:
      \begin{enumerate}
      \item $\exists \bar{\tau} \in ]0, \tau^{*}[$ such that $\upsilon_{G}(\bar{\tau})=\upsilon_{G}(\tau^{*});$
      \item $\upsilon_{G}(\tau^{*}) = \alpha(G)$;
      \item $\forall \bar{\tau} \in ]0,\tau^{*}] \;\; \upsilon_{G}(\bar{\tau}) = \alpha(G).$
      \end{enumerate}
\item $\forall U \subset V(G) \;\; \upsilon_{G-U}(\tau) \le \upsilon_{G}(\tau)$;
\end{enumerate}
\end{theorem}

\begin{proof}
$\;$
\begin{enumerate}
\item Let $\bar{x}$ be the characteristic vector of a maximum stable set of $G$. Since $\bar{x}^TA_G\bar{x}=0$, it follows that
      $$
      2\hat{e}^T\bar{x} - \bar{x}^T\left(\frac{A_G}{\tau} + I_n\right)\bar{x} = 2\hat{e}^T\bar{x} - \|\bar{x}\|^2 = \alpha(G) \le \upsilon_G(\tau).
      $$
\item Assuming that $\tau_1 < \tau_2$, then  $\frac{x^T A_G x}{\tau_2} \le \frac{x^T A_G x}{\tau_1}$ and this inequality is equivalent to the inequality
      $2\hat{e}^T\bar{x} - \frac{\bar{x}^T A_G \bar{x}}{\tau_1} \le 2\hat{e}^T\bar{x} - \frac{\bar{x}^T A_G \bar{x}}{\tau_2}$. Therefore,
      $$
      \upsilon_G(\tau_1) = \max_{x \ge 0} 2\hat{e}^Tx - x^T \left(\frac{A_G}{\tau_1}+I_n\right)x \le \max_{x \ge 0} 2\hat{e}^Tx - x^T\left(\frac{A_G}{\tau_2}+I_n\right)x = \upsilon_G(\tau_2).
      $$
\item Let us prove the implications  $(a) \Rightarrow (b)$, $(b) \Rightarrow (c)$ and  $(c) \Rightarrow (a)$.
      \begin{itemize}
      \item[$\;$] ($(a) \Rightarrow (b)$).
                  Considering $\bar{\tau} \in ]0,\tau^*[$, let $x(\bar{\tau})$ be an optimal solution for $(P_G(\bar{\tau}))$. Then
                  $\upsilon_G(\bar{\tau}) = 2\hat{e}^Tx(\bar{\tau}) - x(\bar{\tau})^T \left(\frac{A_G}{\bar{\tau}}+I_n\right)x(\bar{\tau}) \le 2\hat{e}^Tx(\tau^*) - x(\tau^*)^T \left(\frac{A_G}{\tau^*}+I_n\right)x(\tau^*) \le \upsilon_G(\tau^*)$ and $\upsilon_{G}(\bar{\tau})=\upsilon_{G}(\tau^{*})$ implies
                  $$
                  \frac{x(\bar{\tau})^T A_Gx(\bar{\tau})}{\bar{\tau}} = \frac{x(\tau^*)^T A_G x(\tau^*)}{\tau^*}.
                  $$
                  Therefore, since $\bar{\tau} < \tau^*$ we obtain $x(\bar{\tau})^T A_Gx(\bar{\tau})=0$ which is equivalent to say that the support of $x(\bar{\tau})$ (that is, the subset of indices
                  $\{j: [x(\bar{\tau})]_j >0 \}$) defines the characteristic vector of a vertex subset $S$ which is a stable set of $G$. Since $x(\tau)$ is an optimal solution for $(P_G(\bar{\tau}))$,
                  then $S$ is a maximum stable set and thus $\upsilon_G(\tau^*)=\alpha(G)$.\\
      \item[$\;$] ($(b) \Rightarrow (c)$).
                  Considering $\upsilon_G(\tau^*)=\alpha(G)$, according to the item \eqref{CL2}, $\forall \tau \in ]0, \tau^*]$ it follows that $\alpha(G) \le \upsilon_G(\tau) \le \upsilon_G(\tau^*) = \alpha(G)$.\\
      \item[$\;$] ($(c) \Rightarrow (a)$). This implication is immediate.
      \end{itemize}
\item Let $U$ be a vertex subset of $G$ and let $\bar{x}$ be an optimal solution for $(P_{G-U}(\tau))$. Let $x \in \mathbb{R}^n$ such that
      $$
      x_i = \left\{\begin{array}{ll}
                   \bar{x}_i, & \hbox{if } i \not \in U;\\
                    0, & \hbox{otherwise.}
                   \end{array}\right.
      $$
      Then $\upsilon_{G-U}(\tau) = 2\hat{e}^T x - x^T\left(\frac{A_G}{\tau} + I_n\right)x \le \upsilon_G(\tau)$.
\end{enumerate}
\end{proof}

From the above properties, we may conclude that for any graph $G$, $\upsilon_{G}(\tau)$ is a monotone upper bound on the
stability number of $G$. The Figure~\ref{figura1} illustrates the graph of the function $\upsilon_{G}(\tau),$ obtained for the
graph $G$ depicted in Figure~\ref{figura2}, for which $\alpha(G) = 4$. Both figures appear in \cite{CardosoLuz01}. \\

\begin{figure}[h]
\centerline{\epsfig{file=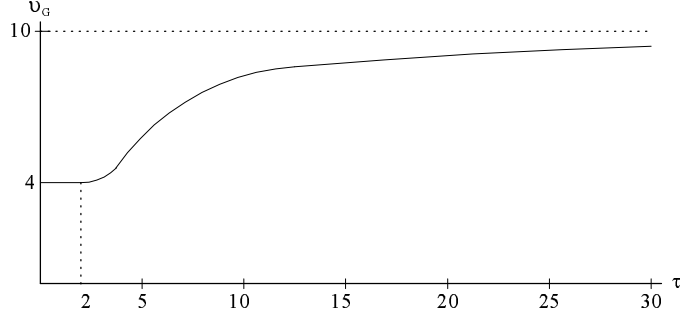,width=9cm}}
\caption{Graph of $\upsilon_{G}(\tau)$, where $G$ is the graph depicted in Figure~\ref{figura2}, for which
         $\alpha(G)=4.$ }\label{figura1}
\end{figure}

\begin{figure}[ht]
\begin{center}
\unitlength=0.25 mm
\begin{picture}(100,150)(-120,40)
%
%
%
\put(-140,110){\circle*{7}} 
\put(-100,110){\circle*{7}} 
\put(-60,110){\circle*{5}}  
\put(-20,110){\circle*{7}}  
\put(20,110){\circle*{7}}   
\put(20,170){\circle*{5}}   
\put(-20,170){\circle*{5}}  
\put(-100,170){\circle*{5}} 
\put(-140,170){\circle*{5}} 
\put(-60,55){\circle*{5}}   
%
\put(-140,100){\makebox(0,0){$a$}}
\put(-100,100){\makebox(0,0){$b$}}
\put(-60,120){\makebox(0,0){$c$}}
\put(-20,100){\makebox(0,0){$d$}}
\put(20,100){\makebox(0,0){$e$}}
\put(20,180){\makebox(0,0){$f$}}
\put(-20,180){\makebox(0,0){$g$}}
\put(-100,180){\makebox(0,0){$i$}}
\put(-140,180){\makebox(0,0){$j$}}
\put(-60,45){\makebox(0,0){$h$}}
%
\put(-140,170){\line(0,-1){60}} 
\put(-140,110){\line(2,3){40}}  
\put(-100,170){\line(0,-1){60}} 
\put(-100,110){\line(-2,3){40}} 
\put(-140,170){\line(1,0){40}}  
\put(-20,110){\line(0,1){60}}   
\put(20,110){\line(0,1){60}}    
\put(20,110){\line(-2,3){40}}   
\put(-20,110){\line(2,3){40}}   
\put(-20,170){\line(1,0){40}}   
\put(-100,110){\line(1,0){80}}  
\put(-60,55){\line(0,1){55}}    
\put(-60,55){\line(3,2){80}}    
\put(-60,55){\line(-3,2){80}}   
\end{picture}
\caption{Graph $G$ of order $10$ with $\upsilon_{G}(2)=\alpha(G)=4$.} \label{figura2}
\end{center}
\end{figure}
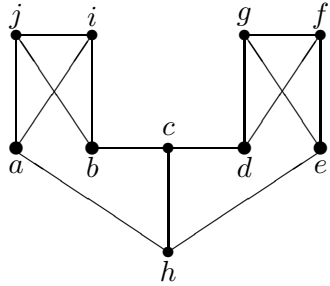

Before to proceed, let us define the family of quadratic programs
\begin{equation}
\left(Q_{G}(\tau)\right) \qquad  \nu_{G}\left(\tau\right) = \min \{z^{T}\left(\frac{A_{G}}{\tau}+I_{n}\right)z: \hat{e}^{T}z=1, z \geq 0 \}, \label{Q_quadratic_family}
\end{equation}
with $\tau>0.$\\

The next theorem which, appears in \cite{CardosoLuz01}, allows to conclude that the indefinite quadratic program of Motzkin-Straus \cite{MotzkinStraus65} is a particular case of the family of quadratic programming problems
($P_{G}(\tau)$).

\begin{theorem}\label{bomze_cardoso_result}
If $x^{*}$ and $z^{*}$ are optimal solutions for $P_{G}(\tau)$ and $Q_{G}(\tau),$ respectively, then $\frac{z^{*}}{\nu_{G}(\tau)}$ and $\frac{x^{*}}{\upsilon_{G}(\tau)}$ are optimal
solutions for $P_{G}(\tau)$ and $Q_{G}(\tau),$ respectively, and $\upsilon_{G}(\tau) = \frac{1}{\nu_{G}(\tau)}.$
\end{theorem}

\begin{proof}
Assume that $x^{*}$ and $z^{*}$ are optimal solutions for $P_{G}(\tau)$ and $Q_{G}(\tau),$ respectively.
Applying the Karush-Khun-Tucker optimality conditions to the quadratic program $(P_{G}(\tau))$, there exists $y^* \ge 0$ such that
\begin{eqnarray}
A_G x^* &=& \tau(\hat{e} - x^*) + y^*, \label{KKT_1}\\
{x^*}^T y^* &=& 0. \label{KKT_2}
\end{eqnarray}
Then ${x^*}^T\left(\frac{A_G}{\tau}+I_n\right)x^*=\hat{e}^Tx^*=\upsilon_G(\tau)$ and thus $\frac{1}{\upsilon_G(\tau)}=\frac{{x^*}^T}{\upsilon_G(\tau)}\left(\frac{A_G}{\tau}+I_n\right)\frac{x^*}{\upsilon_G(\tau)}.$\\
On the other hand, taking into account that $\frac{x^*}{\upsilon_G(\tau)}$ is feasible for $(Q_G(\tau))$, from the optimality of $z^*$ for $(Q_G(\tau))$, we have
${z^*}^T\left(\frac{A_G}{\tau}+I_n\right)z^* \le \frac{{x^*}^T}{\upsilon_G(\tau)}\left(\frac{A_G}{\tau}+I_n\right)\frac{x^*}{\upsilon_G(\tau)}$ and thus
$\upsilon_G(\tau){z^*}^T\left(\frac{A_G}{\tau}+I_n\right)\upsilon_G(\tau)z^* \le {x^*}^T\left(\frac{A_G}{\tau}+I_n\right)x^*$ which is equivalent to
\begin{eqnarray*}
\upsilon_G(\tau) &  =  & 2\hat{e}^Tx^*-{x^*}^T\left(\frac{A_G}{\tau}+I_n\right)x^*\\
                 & \le & 2\hat{e}^T(\upsilon_G(\tau)z^*)-(\upsilon_G(\tau)z^*)^T\left(\frac{A_G}{\tau}+I_n\right)(\upsilon_G(\tau)z^*),
\end{eqnarray*}
since $\hat{e}^T x^* = \upsilon_G(\tau)$ and ${\hat{e}}^T z^* = 1$. Therefore, $\upsilon_G(\tau)z^*$ is an optimal solution for $(P_G(\tau))$ and
$\upsilon_G(\tau) = (\upsilon_G(\tau))^2{z^*}^T\left(\frac{A_G}{\tau}+I_n\right)z^* \Leftrightarrow \upsilon_G(\tau) = \frac{1}{\nu_G(\tau)}$ and
thus $\frac{z^*}{\nu_G(\tau)}$ is an optimal solution for $(P_G(\tau))$.
On the other hand, the optimality of $x^*$ for $(P_G(\tau))$ implies
\begin{eqnarray*}
\upsilon_G(\tau) &  =  & 2\hat{e}^Tx^* -{x^*}^T\left(\frac{A_G}{\tau}+I_n\right)x^*\\
                 & \ge & 2\hat{e}^T(\upsilon_G(\tau)z^*)-(\upsilon_G(\tau)z^*)^T\left(\frac{A_G}{\tau}+I_n\right)(\upsilon_G(\tau)z^*).
\end{eqnarray*}
Since $\hat{e}^Tx^* = \hat{e}^T(\upsilon_G(\tau)z^*)$, we have
$$(\upsilon_G(\tau)z^*)^T\left(\frac{A_G}{\tau}+I_n\right)(\upsilon_G(\tau)z^*) \ge {x^*}^T\left(\frac{A_G}{\tau}+I_n\right)x^*$$
which is equivalent to $\nu_G(\tau) = {z^*}^T\left(\frac{A_G}{\tau}+I_n\right)z^*  \ge
           \frac{{x^*}^T}{\upsilon_G(\tau)}\left(\frac{A_G}{\tau}+I_n\right)\frac{x^*}{\upsilon_G(\tau)} = \frac{1}{\upsilon_G(\tau)}$.
\end{proof}

The proof of Theorem~\ref{bomze_cardoso_result} could be obtained applying Theorem 5 in  \cite{Bomze} to the quadratic programming problems
$(P_G(\tau))$ and $(Q_G(\tau))$.\\

Considering an arbitrary graph $G$ of order $n$, Motzkin and Straus in \cite{MotzkinStraus65} proved the following result.

\begin{theorem}\label{Theorem_motzkin_straus}\cite{MotzkinStraus65}
If $G$ is a graph of order $n$ and $\Delta = \{ x \ge 0: \sum_{j=1}^n{x_j} = 1 \}$, then
$\max_{x \in \Delta} x^{T}A_{G}x = 1 - \frac{1}{\omega(G)}.$
\end{theorem}

Assuming that the graph $G$ has order $n$ and size $m$, from Teorema~\ref{Theorem_motzkin_straus}, setting $x_i=1/n$, for $i=1, \ldots, n$,
it follows
$
1 - \frac{1}{\omega(G)} \ge \frac{2m}{n^2}.
$
Therefore, in the particular case of triangle-free graphs (graphs without triangles) we obtain the inequality $m \le \frac{n^2}{4}$. This upper
bound on the size of $G$ coincides with the extremal value on the number of edges of triangle-free graphs given by the classical Mantel's theorem
and extended by Tur\'an who, with the well known Turan's graph theorem \cite{Turan41}, started the Theory of Extremal Graphs.

\begin{theorem}\cite{Turan41}
If $G$ is a $K_q$-free graph, with $q>1$, then $m \le \frac{(q-2)}{2(q-1)}n^2$.
\end{theorem}

Taking into account the Motzkin-Straus quadratic model for the determination of the clique number of a graph
(Theorem~\ref{Theorem_motzkin_straus}), it follows that
\begin{eqnarray*}
\frac{1}{2}(1-\frac{1}{\omega(G)}) &=& \max_{x \in \Delta} \frac{1}{2}x^{T}A_{G}x = \max_{x \in \Delta} \sum_{ij \in E(G)}{x_ix_j} =
                                       \max_{x \in \Delta} \sum_{ij \in E(K_n)}{x_ix_j} - \sum_{rs \in E(\overline{G})}{x_rx_s}\\
                                   &=& \max_{x \in \Delta} \frac{1}{2}[(\sum_{j=1}^{n}{x_j})^2 - ||x||^2] - \sum_{rs \in E(\overline{G})}{x_rx_s} =
                                       \max_{x \in \Delta} \frac{1}{2}(1-||x||^2 - x^{T}A_{\overline{G}}x)\\
                                   &=& \frac{1}{2} - \min_{x \in \Delta} \frac{1}{2} x^{T}(A_{\overline{G}}+I_n)x
\end{eqnarray*}
and then
$
\min_{x \in \Delta} x^{T}(A_{\overline{G}}+I_n)x = \frac{1}{\omega(G)}
                       \Leftrightarrow  \min_{x \in \Delta} x^{T}(A_{\overline{G}}+I_n)x = \frac{1}{\alpha(\overline{G})}
$
Therefore, the indefinite Motzkin-Straus quadratic model for the determination of the clique number of the complement
$\overline{G}$ of a graph $G$ is equivalent to the indefinite quadratic model for the determination of the stability number
\begin{eqnarray}
\frac{1}{\min\{x^{T}(A_{G}+I_{n})x: \hat{e}^{T}x=1, x \geq 0\}} = \alpha(G). \label{motzkin_straus}
\end{eqnarray}
(see\cite[Prop. 2]{Gibbons_et_al97}). Combining \eqref{motzkin_straus} with Theorem~\ref{bomze_cardoso_result} we have
\begin{equation}
\upsilon_{G}(1) = \frac{1}{\nu_{G}(1)} = \alpha(G). \label{relations_with_MS}
\end{equation}

As a consequence of \eqref{relations_with_MS}, if $x^*$ is an optimal solution of $(P_G(\tau))$, with $\tau \ge 1$, then
$$
2\hat{e}^Tx^* - {x^*}^T\left(A_G+I_n\right)x^* \le \upsilon_G(1) = \alpha(G) \le \upsilon_G(\tau),
$$
and since $\upsilon_G(\tau) = \hat{e}^Tx^* = {x^*}^T\left(\frac{A_G}{\tau}+I_n\right)x^*$, it follows that
$\upsilon_G(\tau) - \frac{\tau-1}{\tau}{x^*}^T A_G x^* \le \alpha(G) \le \upsilon_G(\tau)$. On the other hand, from \eqref{motzkin_straus},
$\frac{1}{\frac{{x^*}^T}{\upsilon_G(\tau)}\left(A_G+I_n\right)\frac{x^*}{\upsilon_G(\tau)}} \le \alpha(G)$. Therefore,
\begin{eqnarray}
||x^{*}||^{2} - (\tau - 2)(\upsilon_{G}(\tau) - ||x^{*}||^{2}) \leq &\! \alpha(G)  & \leq \upsilon_{G}(\tau) \label{lower_bound_1}
\end{eqnarray}
and
\begin{eqnarray}
\frac{\upsilon_{G}(\tau)^{2}}{x^{*T}(A_{G}+I_{n})x^{*}} \leq \!& \alpha(G) & \leq \upsilon_{G}(\tau). \label{lower_bound_2}
\end{eqnarray}

The lower bounds in (\ref{lower_bound_1}) and (\ref{lower_bound_2}) were obtained in \cite{Cardoso01} and \cite{CardosoLuz01}, respectively.

\section{Characterization and properties of convex-$QP$ graphs}\label{Convex-QP}

As already referred in the last section, considering a graph $G$ with at least one edge, when $\tau \ge -\lambda_{min}(G)$  the quadratic
program $(P_G(\tau))$ is convex. Since $\upsilon_G(\tau)$ is non decreasing on the variable $\tau$, with the aim of obtaining a polynomial-time
upper bound on the stability number, does not make sense to consider values of $\tau$ greater than $-\lambda_{min}(G)$ (notice that
$\lambda_{min}(G)$ is negative for graphs with at least one edge). The convex quadratic upper bound $\upsilon_{G}(-\lambda_{min}(G))$ was
firstly introduced in 1995 by Luz in \cite{Luz95}. From now on, $(P_G(-\lambda_{min}(G))$ and its optimal value $\upsilon_{G}(-\lambda_{min}(G))$
will be simple denoted by $(P(G))$ and $\upsilon(G)$, respectively, that is,
\begin{equation}
(P(G)) \qquad \upsilon(G) = \max \{2\hat{e}^{T}x - x^T\left(H + I\right)x: x \ge 0\}, \label{Luz_UpperBound}
\end{equation}
where $H = \frac{A_G}{-\lambda_{min}(G)}$.\\

It is immediate that $H$ is a positive semidefinite matrix and thus the quadratic program \eqref{Luz_UpperBound} is convex. According to
Theorem~\ref{cardoso_luz_th}-\eqref{CL1}, $\alpha(G) \le \upsilon(G)$, that is, $\upsilon(G)$ is convex quadratic upper bound on the stability
number of $G$. When $G$ is a regular graph of order $n$ with at least one edge, as proved in \cite{Luz95, LuzCardoso1},
$\upsilon(G) = n\frac{-\lambda_n(G)}{\lambda_1(G)-\lambda_n(G)}$ which is the  very popular upper bound on the stability number of regular graphs
obtained by Hoffman (unpublished) and presented by Lov\'asz in  \cite{Lovasz79} as follows.
\begin{equation}
\alpha(G) \le n\frac{-\lambda_n(G)}{\lambda_1(G)-\lambda_n(G)}. \label{Hoffman_bound}
\end{equation}

Now we introduce a necessary and sufficient condition, proved in \cite{Luz95}, for the the stability number of a graph be attained by the upper bound \eqref{Luz_UpperBound}.

\begin{theorem}\cite{Luz95}
Let $G$ be a graph with at least one edge. Then $\alpha(G) = \upsilon(G)$ if and only if for a maximum stable set $S$ of $G$
(and then for all),
\begin{eqnarray}
-\lambda_{min}(G) & \le & \min \{|N_G(i) \cap S|: i \not \in S\}. \label{Luz_Condition}
\end{eqnarray}
\end{theorem}

A slight modified version is the following.

\begin{theorem}\cite{CardosoCvetkovic06}\label{cardoso_cvetkovic_th}
Let $G$ be a graph with at least one edge. Then $\alpha(G) = \upsilon(G)$ if and only if there exists a stable set $S$ for which
\eqref{Luz_Condition} holds.
\end{theorem}

From Theorem~\ref{cardoso_cvetkovic_th}, we may conclude that if a stable set $S$ with the property \eqref{Luz_Condition} is found,
then $S$ is a maximum stable set.\\

A graph $G$ such that $\alpha(G)=\upsilon(G)$ was called in \cite{Cardoso01} graph with convex quadratic stability number or simply
graph with convex-$QP$ stability number (where $QP$ stands for quadratic programming). The class of these graphs with convex-$QP$
stability number is denoted by $\mathcal{Q}$ and a graph $G \in \mathcal{Q}$ is called a $\mathcal{Q}$-graph.\\

A class of graphs is hereditary when it is closed under vertex deletion. Thus, if a graph $G$ belongs to some hereditary class,
then $G-\{v\}$ also belongs to the same class for any vertex $v \in V(G)$. The class of graphs $\mathcal{Q}$ is not hereditary
\cite{CardosoLozin12}. However, according to \cite[Th. 3]{Cardoso01}, this class is closed under deletion of $\alpha(G)$-redundant
vertices, that is, vertices $v \in V(G)$ such that $\alpha(G) = \alpha(G-\{v\})$. More generally, considering a $\alpha(G)$-redundant
vertex subset $U$ (that is, $\alpha(G) = \alpha(G-U)$), we have the following result.

\begin{theorem}\cite{Cardoso01}
Let $G$ be a graph and $U \subseteq V(G)$ an $\alpha(G)$-redundant vertex subset. If $G \in \mathcal{Q}$, then $G-U \in \mathcal{Q}$.
\end{theorem}

The class of graphs $\mathcal{Q}$ is infinite. For instance all the line graphs of connected graphs with a perfect matching are in $\mathcal{Q}$
as it is stated by the following theorem.
F
\begin{theorem}\cite{Cardoso01}
A connected graph $G$ with at least one edge, which is nor a star neither a triangle, has a perfect matching if and only if its line graph
$L(G)$ is a $\mathcal{Q}$-graph.
\end{theorem}

According to Las Vergnas \cite{LasVergnas75}, every connected claw-free graph of even order has a perfect matching. Since the line graphs
are claw-free, every line graph of a connected graph with even size has a perfect matching. As immediate consequence, we have the
following corollary.

\begin{corollary}\cite{Cardoso01}\\
If $G$ is a connected graph with an even size, then $L(L(G))$ is a $\mathcal{Q}$-graph.
\end{corollary}

There are several famous $\mathcal{Q}$-graphs as it is the case of the Petersen graph $P$, depicted in Figure~\ref{petersengraph},
for which $\alpha(P) = \upsilon(P) = 4$.

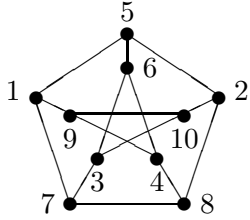
\begin{figure}[h]
\begin{center}
\unitlength=0.3 mm
\begin{picture}(400,100)(50,90)
%
\put(235,110){\line(2,1){55}}  
\put(265,110){\line(-2,1){55}} 
\put(250,150){\line(-1,-3){14}}
\put(250,150){\line(1,-3){14}} 
\put(225,90){\line(1,0){50}}   
\put(225,90){\line(-1,3){16}}  
\put(225,90){\line(3,5){12}}   
\put(275,90){\line(1,3){16}}   
\put(275,90){\line(-3,5){12}}  
\put(225,130){\line(1,0){50}}  
\put(250,165){\line(-3,-2){40}}
\put(250,165){\line(0,-1){15}} 
\put(250,165){\line(3,-2){40}} 
%
\put(237,110){\circle*{5.7}} 
\put(263,110){\circle*{5.7}} 
\put(225,90){\circle*{5.7}} 
\put(275,90){\circle*{5.7}} 
\put(250,165){\circle*{5.7}}
\put(250,150){\circle*{5.7}}
\put(210,137){\circle*{5.7}}
\put(290,137){\circle*{5.7}}
\put(225,129){\circle*{5.7}}
\put(275,129){\circle*{5.7}}
%
%
\put(237,100){\makebox(0,0){3}}
\put(263,100){\makebox(0,0){4}}
\put(225,119){\makebox(0,0){9}}
\put(275,119){\makebox(0,0){10}}
\put(260,150){\makebox(0,0){6}}
\put(215,90){\makebox(0,0){7}}
\put(285,90){\makebox(0,0){8}}
\put(250,175){\makebox(0,0){5}}
\put(200,140){\makebox(0,0){1}}
\put(300,140){\makebox(0,0){2}}
\end{picture}
\end{center}
\caption{The Petersen graph.}\label{petersengraph}
\end{figure}

According to \cite{BarbosaCardoso04}, a graph $G$ is $\tau$-regular-stable if there exists a maximum stable set $S$ such that
\begin{equation}
\forall v \in V(G) \setminus S, \qquad |N_{G}(v) \cap S| = \tau. \label{tau_regular_stable}
\end{equation}
For instance, the graph depicted in Figure~\ref{figura2} and the Petersen graph depicted in Figure~\ref{petersengraph}, are
$2$-regular-stable. The $\tau$-regular-stable graphs are particular cases of $\mathcal{Q}$-graphs, when $\tau = -\lambda_{min}(G)$
(as it is the case of the above referred graphs). According to \cite{BarbosaCardoso04}, if a graph $G$ is $\tau$-regular-stable, then
$$
\frac{n\tau}{\Delta(G)+\tau} \leq \alpha(G) \leq \frac{n\tau}{\delta(G)+\tau},
$$
where $\Delta(G)$ and $\delta(G)$ denote, respectively, the maximum and minimum degree of the vertices of $G$. Therefore, if $G$ is
$p$-regular and $\tau$-regular-stable then $\alpha(G) = \frac{n\tau}{p+\tau}.$ \\

It is immediate that a graph is in $\mathcal{Q}$ if and only if each of its components belongs to $\mathcal{Q}$. On the other hand,
every graph $G$ has a subgraph $H$ belonging to $\mathcal{Q}$ and such that $\alpha(G)=\alpha(H)$. In the worst case, deleting as many
vertices not belonging to a maximum stable set as necessary, a graph with the same stability number and in which every component is a
clique is obtained and this graph is a $\mathcal{Q}$-graph.\\

Now it is worth mention the concept of $(\kappa,\tau)$-regular set, introduced in \cite{CardosoRama04}, which is a vertex subset $S$
of a graph $G$, inducing a $\kappa$-regular subgraph such that every vertex out of $S$ has $\tau$ neighbors in $S$, that is, for any
vertex $v \in V(G)$ we have
$$
|N_G(v) \cap S| = \left\{\begin{array}{ll}
                                \kappa, & \hbox{if } v \in S;\\
                                \tau,   & \hbox{otherwise.}
                         \end{array}\right.
$$
For convenience, when $G$ is a $p$-regular graph, the whole vertex set $V(G)$ is considered a $(p,0)$-regular set. For instance,
considering the Petersen graph depicted in Figure~\ref{petersengraph}, the following $(\kappa,\tau)$-regular sets are obtained.
\begin{itemize}
\item{The set $S_1=\{1,2,3,4\}$ is $(0,2)$-regular.}
\item{The set $S_2=\{5,6,7,8,9,10\}$ is $(1,3)$-regular.}
\item{The set $S_3=\{1, 2, 5, 7, 8 \}$ is $(2,1)$-regular.}
\end{itemize}
Notice that a $\tau$-regular-stable graph defined by the condition \eqref{tau_regular_stable} is a graph with a $(0,\tau)$-regular set.
A nice property of a $p$-regular graph $G$ with a $(\kappa,\tau)$-regular set is that $\lambda = \kappa-\tau$ is an eigenvalue of $G$.
By definition, the whole vertex set of $G$ is $(p,0)$-regular and then $p-0=p$ is an eigenvalue of $G$ (in fact it is the largest eigenvalue
of $G$). \\

Using the concept of $(\kappa,\tau)$-regular set, we may introduce the following necessary and sufficient condition for a regular graph
be a $\mathcal{Q}$-graph.

\begin{theorem}\cite{CardosoCvetkovic06}
Let $G$ be a regular graph with at least one edge. Then $G$ is a $\mathcal{Q}$-graph if and only if there exists a $(0,\tau)$-regular
set $S \subset V(G)$, with $\tau=-\lambda_{min}(G)$. Furthermore, $S$ is a maximum stable set and every maximum stable set is $(0,\tau)$-regular.
\end{theorem}

Regarding the existence of $(\kappa,\tau)$-regular sets in regular graphs, it follows a necessary and sufficient condition deduced
in \cite{Thompson81} (using a different terminology).

\begin{theorem} \cite{Thompson81}
A $p$-regular graph $G$ has a $(\kappa,\tau)$-regular set $S \subset V(G)$ if and only if $\kappa-\tau$ is an eigenvalue and
$\hat{u} = \mathbf{x} -\frac{\tau}{p+\tau-\kappa}\hat{e}$, where $\mathbf{x}$ is the characteristic vector of $S$, is an eigenvector
associated to $\lambda=\kappa-\tau$.
\end{theorem}

\section{Recognition of convex-$QP$ graphs}\label{Recognition}

This section starts with a theorem that summarizes some results obtained in \cite{Cardoso01} with the purpose of designing an algorithm
for the recognition of $\mathcal{Q}$-graphs.

\begin{theorem}\label{algorithmic_results} \cite{Cardoso01}
Let $G$ be a graph with at least one edge.
\begin{enumerate}
\item Assuming that $\lambda_{min}(G) < \lambda_{min}(G-U),$ with $U \subset V(G)$,
      \begin{eqnarray*}
      \upsilon(G) = \upsilon(G-U) & \Rightarrow & G \in {\cal Q};\\
      \upsilon(G) > \upsilon(G-U) & \Rightarrow & G \not \in {\cal Q} \mbox{ or } \alpha(G-U)<\alpha(G).
      \end{eqnarray*}
\item If $\exists v \in V(G)$ such that $\upsilon(G) \not = \max\{\upsilon(G-\{ v \}), \upsilon(G-N_{G}(v))\}$, then $G \not \in {\cal Q}.$
\item Consider that $\exists v \in V(G)$ such that $\upsilon(G-\{ v \}) \not= \upsilon(G-N_{G}(v))$.
      \begin{enumerate}
      \item{If $\upsilon(G) = \upsilon(G-\{ v \})$ then $G \in {\cal Q} \mbox{ iff } G-\{v\} \in {\cal Q}.$}
      \item{If $\upsilon(G) = \upsilon(G-N_{G}(v))$ then $G \in {\cal Q} \mbox{ iff } G-N_{G}(v) \in {\cal Q}.$}
      \end{enumerate}
\end{enumerate}
\end{theorem}

Notice that $\lambda_{min}(G)<0$ when $G$ has at least one edge and, after the deletion of vertices, the least eigenvalue does not decrease
and the optimal value of \eqref{Luz_UpperBound} does not increase, that is,
$$
\forall U \subset V(G), \;\; \left\{\begin{array}{rcl}
                                          \lambda_{min}(G) & \le & \lambda_{min}(G-U),\\
                                          \upsilon(G)      & \ge & \upsilon(G-U).
                                   \end{array}\right.
$$

Applying the results of Theorem~\ref{algorithmic_results}, an algorithm proposed in \cite[Alg. 1]{CardosoLuz16} recognizes in polynomial-time
if a given graph $G$ belongs or not to $\mathcal{Q}$ or, if none of these conclusions is possible, identifies a so called adverse subgraph of
$G$. Furthermore, according to the Algorithm 1 in \cite{CardosoLuz16}, if using some additional procedure we are able to conclude that the
obtained adverse subgraph is a $\mathcal{Q}$-graph, then the graph $G$ is a $\mathcal{Q}$-graph. The Algorithm 1 in \cite{CardosoLuz16} is
similar to the algorithm presented in \cite{Cardoso01} for the case of the recognition of line graphs belonging to $\mathcal{Q}$. An adverse
graph (a concept introduced in \cite{CardosoLuz01}) is a graph $H$ without isolated vertices with the following properties:
\begin{enumerate}
\item $\upsilon(H)$ and $\lambda_{\min}(H)$ are integers;
\item For any vertex $i \in V(H), \;\; \left\{\begin{array}{l}
                                             \upsilon(H-N_{H}(i))=\upsilon(H);\\
                                             \lambda_{\min}(H-N_{H}(i))=\lambda_{\min}(H).
                                             \end{array}\right.$
\end{enumerate}
Notice that if $H$ is an adverse graph, the equalities $\upsilon(H-i)=\upsilon(H)$ and $\lambda_{\min}(H-\{i\})=\lambda_{\min}(H)$ also hold
for each vertex $i\in V(H).$ Indeed, since $H$ has no isolated vertices, for any vertex $i\in V(H)$, there exists a vertex $j$ such that
$i\in N_{H}(j)$. Therefore, taking into account the above definition and the inequalities%
\[
\upsilon(H-N_{H}(j))\leq \upsilon(H-i)\leq \upsilon(H),
\]
the equality $\upsilon(H-i)=\upsilon(H)$ holds. By similar arguments, the equalities $\lambda_{\min}(H-i)=\lambda_{\min}(H)$ for every
vertex $i\in V(H)$ also hold.\\

The Petersen graph in Figure~\ref{petersengraph} is an example of an adverse graph.\\

In \cite{CardosoLuz16} the  following conjecture was posed.

\begin{conjecture}\cite{CardosoLuz16}\label{conjecture}
All adverse graphs are $\mathcal{Q}$-graphs.
\end{conjecture}

Notice that all performed computational tests with adverse graphs supported the validity of Conjecture~\ref{conjecture}, however the
conjecture remains open.\\

From the above analysis it follows that the polynomial-time recognition of $\mathcal{Q}$-graphs is equivalent to the existence of a
polynomial-time algorithm for deciding whether an adverse graph is a $\mathcal{Q}$-graph and so far this problem remains open. \\

There are several families of graphs in which we may decide in polinomial-time whether a graph is a $\mathcal{Q}$-graph or not, as it
is the case of bipartite graphs, since these graphs have no adverse subgraphs (when a vertex is deleted the least eigenvalue increases).
Some other families of this type like threshold graphs, $(C_4,P_5)$-free  graphs, $(K_{1,3},P_5)$-free graphs, etc., appear in \cite{CardosoLuz01, Cardoso03}.\\

As previously mentioned, in general, for the recognition of $\mathcal{Q}$-graphs we may apply the Algorithm 1 in \cite{CardosoLuz16} to
an arbitrary graph $G$ and the main obstacle to such recognition is when this algorithm identifies an adverse subgraph $H$. In such a case,
if $H$ is a $\mathcal{Q}$-graph, then $G$ is also $\mathcal{Q}$-graph. The next result deduced in \cite{Cardoso03} relates adverse
$\mathcal{Q}$-graphs with the existence of $(0,\tau)$-regular sets.

\begin{theorem}\cite{Cardoso03}\label{cardoso_condition}
Let $G$ be an adverse graph and $\tau = -\lambda_{min}(G)$. Then $G \in {\cal Q}$ if and only if $\exists S \subset V(G)$
which is $(0,\tau)$-regular.
\end{theorem}

As immediate consequence of Theorem \ref{cardoso_condition}, deciding whether an adverse graph $G$ is a $\mathcal{Q}$-graph
is equivalent to decide if it has a $(0,\tau)$-regular set, with $\tau=-\lambda_{min}(G)$.\\

The next subsections analyse two approaches to the determination of $(0,\tau)$-regular sets in adverse graphs. The first approach
can be applied to the general problem of determining $(\kappa,\tau)$-regular sets by using specialized algorithms to obtain binary 0-1
solutions of linear systems. The second approach is based on the theory of star sets/star complements and applies specifically to
the determination of $(0,\tau)$-regular sets, taking into account that their characteristic vectors are vertices of the polytope
of nonnegative solutions of the linear system \eqref{k_t_system} below.

\subsection{Determination of a $(0,\tau)$-regular set of an adverse graph by solving a linear system}

The following theorem unifies the results obtained in \cite{CardosoRama04, CardosoLozinLuzPacheco16} which can be applied
for the determination of a $(\kappa,\tau)$-regular set, with $\kappa=0$ and $\tau=-\lambda_{min}(G)$.

\begin{theorem}\cite{CardosoRama04, CardosoLozinLuzPacheco16}
Considering a graph $G$ of order $n$, let $\overline{x}$ be a particular solution of the linear system
\begin{equation}
 \begin{array}{rcl}
 \left(A_G - (\kappa-\tau)I_{n}\right)x & = & \tau \hat{e}.\label{k_t_system}
 \end{array}
\end{equation}
Then, the graph $G$ has a $(\kappa,\tau)$-regular set $S \subset V(G)$ if and only if one of the following conditions hold.
\begin{enumerate}
\item The system \eqref{k_t_system} has a $0-1$ solution which is the characteristic vector of $S$.
\item The characteristic vector of $S$, $\mathbf{x}$, is such that
      \begin{equation}
      \mathbf{x}=\overline{x}+\hat{u},
      \end{equation}
      where $\hat{u}=\mathbf{0}$ if $\kappa-\tau$ is not an eigenvalue of $G$ and $\hat{u}$ is an eigenvector associated to the
      eigenvalue $\lambda = \kappa-\tau$, otherwise.
\end{enumerate}
\end{theorem}

\begin{corollary}\cite{CardosoLozinLuzPacheco16}
If a graph $G$ has a $(\kappa,\tau)$-regular set $S \subseteq V(G)$ and $\overline{x}$  is a particular solution of \eqref{k_t_system}, then
$|S| = \hat{e}^{T}\overline{x}$.
\end{corollary}

Therefore, considering any particular solution $\overline{x}$ of \eqref{k_t_system}, if $\hat{e}^{T}\overline{x}$ is not integer,
then there is no $(\kappa,\tau)$-regular set in $G$ \cite{CardosoLozinLuzPacheco16}.\\

\subsection{A simplex-like approach to the determination of a $(0,\tau)$-regular set of an adverse graph based on star sets/star complements}\label{simplex_subsection}

A simplex-like algorithm for the recognition of adverse graphs which are $\mathcal{Q}$-graphs was introduced in \cite{CardosoLuz16}.
This algorithm deals with bases defined by star complements \cite{CvetRowSim2010} of a graph $G$ associated to its least eigenvalue
$\lambda_{min}(G)$.\\

Let us assume that a graph $G$ of order $n$ has $q$ distinct eigenvalues $\mu_1, \mu_2, \dots, \mu_q$, where $\mu_1$ has multiplicity
$m_1$, $\mu_2$ multiplicity $m_2$, $\cdots$,  $\mu_q$  multiplicity $m_q$ (and then $m_1+m_2+ \cdots + m_q = n$).  A vertex subset $X$
of $G$ is a star set for some eigenvalue $\mu_i(G)$ with multiplicity $m_i$, if $|X|=m_i$ and $\mu_i(G)$ is not an eigenvalue
of $G-X$. Additionally, if $X$ is a star set of $G$ for the eigenvalue $\mu_i(G)$, then $V(G) \setminus X$ is the co-star
set $\overline{X}$ and the subgraph $G-X = G[\overline{X}]$ is a star complement of $G$ for $\mu_i(G)$. Denoting by $A_X$ the submatrix
of $A_G$ corresponding to the subgraph $G[X]$ and by $C_{\bar{X}}$ the submatrix of $A_G$ corresponding to the subgraph $G[\overline{X}]$,
assuming an adequate permutation of the vertices, throughout the remaining part of the text the adjacency matrix $A_G$ is considered
with the block partition
$$
A_G = \left[\begin{array}[c]{cc}%
                  A_{X} & N^{T}\\
                  N     & C_{\bar{X}}%
             \end{array}\right].
$$
The vertex set $V(G)$ admits a star partition into the star sets $X_1, X_2, \dots, X_q$, where $X_1$ is a star set for $\mu_1$ (and then $|X_1|=m_1$),
$X_2$ is a star set for $\mu_2$ (and then $|X_2|=m_2$), $\cdots$, $X_q$ is a star set for $\mu_q$ (and then $|X_q|=m_q$). In Figure~\ref{star_partition}, the labels of the vertices
of the graph depicted on the right are the eigenvalues of $G$ and the vertices with the same label form a star set for the corresponding
eigenvalue. Notice that the graph $G$ has the eigenvalue $-2$ with multiplicity $2$, the eigenvalue $-1$ with multiplicity $1$, the eigenvalue
$0$ with multiplicity $1$, the eigenvalue $1$ with multiplicity $2$ and the eigenvalue $3$ with  multiplicity $1$. The star partition
depicted in Figure~\ref{star_partition} is $V(G)=\{a, d\} \cup \{b, c\} \cup \{e\} \cup \{f\} \cup \{g\}$.


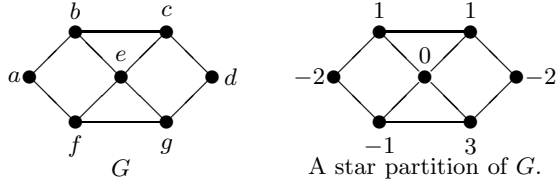
\begin{figure}[ht]
          \begin{center}
          \unitlength=0.4 mm
          \begin{picture}(205,60)(-50,100)
          %
          \put(-30,135){\makebox(0,0){\footnotesize $a$}}
          \put(-10,157){\makebox(0,0){\footnotesize $b$}}
          \put(20,157){\makebox(0,0){\footnotesize $c$}}
          \put(41,135){\makebox(0,0){\footnotesize $d$}}
          \put(20,112){\makebox(0,0){\footnotesize $g$}}
          \put(-10,112){\makebox(0,0){\footnotesize $f$}}
          \put(5,142){\makebox(0,0){\footnotesize $e$}}
          \put(-25,135){\circle*{4}} 
          \put(-10,150){\circle*{4}} 
          \put(20,150){\circle*{4}} 
          \put(35,135){\circle*{4}} 
          \put(20,120){\circle*{4}} 
          \put(-10,120){\circle*{4}} 
          \put(5,135){\circle*{4}} 
          \put(-10,120){\line(1,0){30}} 
          \put(-10,150){\line(1,0){30}} 
          \put(-10,120){\line(-1,1){15}} 
          \put(-10,150){\line(-1,-1){15}} 
          \put(20,120){\line(1,1){15}} 
          \put(20,150){\line(1,-1){15}} 
          \put(-10,120){\line(1,1){30}} 
          \put(20,120){\line(-1,1){30}} 
          \put(5,105){\footnotesize \makebox(0,0){$G$}}
          %
          %
          %
          \put(67,135){\makebox(0,0){\footnotesize $-2$}}
          \put(90,157){\makebox(0,0){\footnotesize $1$}}
          \put(120,157){\makebox(0,0){\footnotesize $1$}}
          \put(143,135){\makebox(0,0){\footnotesize $-2$}}
          \put(120,112){\makebox(0,0){\footnotesize $3$}}
          \put(90,112){\makebox(0,0){\footnotesize $-1$}}
          \put(105,142){\makebox(0,0){\footnotesize $0$}}
          \put(75,135){\circle*{4}} 
          \put(90,150){\circle*{4}} 
          \put(120,150){\circle*{4}} 
          \put(135,135){\circle*{4}} 
          \put(120,120){\circle*{4}} 
          \put(90,120){\circle*{4}} 
          \put(105,135){\circle*{4}} 
          \put(90,120){\line(1,0){30}} 
          \put(90,150){\line(1,0){30}} 
          \put(90,120){\line(-1,1){15}} 
          \put(90,150){\line(-1,-1){15}} 
          \put(120,120){\line(1,1){15}} 
          \put(120,150){\line(1,-1){15}} 
          \put(90,120){\line(1,1){13}} %
          \put(120,150){\line(-1,-1){13}} %
          \put(120,120){\line(-1,1){13}} %
          \put(90,150){\line(1,-1){13}} %
          \put(105,105){\makebox(0,0){\footnotesize A star partition of $G$.}}
          \end{picture}
          \caption{A graph $G$ and a star partition of $G$.} \label{star_partition}
          \end{center}
          \end{figure}

\begin{theorem}\cite{CvetkovicRowlinson04}
Every graph has a star partition.
\end{theorem}

In general, each graph has several star partitions. For instance, the Petersen graph (depicted in Figure~\ref{petersengraph}) has $750$ star partitions.
Furthermore, if $X$ is a star set for the eigenvalue $\mu_i(G)$, then there exists a star partition $V(G)=X_1 \cup \cdots \cup X_i \cup \cdots \cup X_q$ such that $X=X_i$.
Not every vertex of a graph $G$ belongs to some star set of an eigenvalue $\mu_j(G)$. However, every vertex belongs to some co-star set of $\mu_j(G)$.

\begin{lemma}\cite{CardosoLozinLuzPacheco16}\label{subsystem_lemma}
Let $G$ be a graph of order $n$, $\lambda$ an eigenvalue of $G$ and $X \subset V(G)$ a star set of $\lambda$. Then the rows of the submatrix
\begin{equation}\label{star_complement_submatrix}
  \left[\begin{array}[c]{cc}%
   N \; & \; C_{\bar{X}}-\lambda I_{\bar{X}}%
  \end{array}\right]
\end{equation}
spans the row space of $A_G - \lambda I_n = \left[\begin{array}[c]{cc}%
                                          A_{X}-\lambda I_{X} & N^{T}\\
                                          N                   & C_{\bar{X}}-\lambda I_{\bar{X}}%
                                         \end{array}\right].$
Moreover, $X' \subset V(G)$ is another star set of $\lambda$ if and only if the submatrix of \eqref{star_complement_submatrix} defined by the
columns indexed by the vertices in the co-star set $\overline{X}'$ is basic, that is, nonsingular.
\end{lemma}

As a consequence of Lemma~\ref{subsystem_lemma}, the linear system \eqref{k_t_system} is equivalent to the linear subsystem with just the equations
corresponding to the rows of the matrix \eqref{star_complement_submatrix}, that is, is equivalent to the linear subsystem
\begin{eqnarray}
\left[\begin{array}[c]{cc}%
   N \; & \; C_{\bar{X}}-\lambda I_{\bar{X}}%
\end{array}\right] x & = & \tau \hat{e}{\bar{X}}, \label{sisreduzido}
\end{eqnarray}
where the all one vector $\hat{e}_{\bar{X}}$ has $|\bar{X}|$ components.\\

The next result allows the search for characteristic vectors of $(0,\tau)$-regular sets, that is, $0-1$ solutions of the linear system \eqref{k_t_system},
by application of the simplex method to the subsystem \eqref{sisreduzido}, where $\lambda=-\tau$.

\begin{theorem}\cite{CardosoLozinLuzPacheco16}
Every $0-1$ solution of the linear system \eqref{k_t_system}, with $\lambda=-\tau$,  is a basic nonnegative solution of the subsystem \eqref{sisreduzido}.
\end{theorem}

Notice that when $\lambda = -\tau$ is not an eigenvalue of $G$, the linear system \eqref{k_t_system} has an unique solution and then $G$ has a $(0, \tau)$-regular
set if and only if this unique solution is $0-1$. When  $\lambda=-\tau$ is an eigenvalue of $G$, since the system \eqref{k_t_system} has the same set of solutions
as the subsystem \eqref{sisreduzido}, $G$ has a $(0, \tau)$-regular set if and only if there is a $0-1$ basic solution among the basic solutions of \eqref{sisreduzido}.\\

If $\overline{X}$ is some co-star set of $-\tau = \lambda_{min}(G)$ and $x=\left[\begin{array}{c}
                                                                                        x_N\\
                                                                                        x_B
                                                                                  \end{array}\right],$ with $x_N=0$, is a basic solution of \eqref{sisreduzido}, where
$\lambda=-\tau$, then multiplying both sides of \eqref{sisreduzido} by $\left(C_{\bar{X}}+\tau I_{\bar{X}}\right)^{-1}$ (the inverse of the basic submatrix) we obtain
\begin{eqnarray*}
\left[\begin{array}{cc}
       \left(C_{\bar{X}}+\tau I_{\bar{X}}\right)^{-1}N \; & \; I_{\bar{X}}\\
      \end{array}\right] \left[\begin{array}{c}
                                      x_N\\
                                      x_B
                               \end{array}\right] &=&
                               \tau \left(C_{\bar{X}}+\tau I_{\bar{X}}\right)^{-1}e_{\bar{X}}.
\end{eqnarray*}
Therefore, the reduced simplex tableau for this basic solution is the tableau

\begin{center}
\begin{equation}\label{simplex_tableau}
\begin{tabular}{c|c|c}
                 & $x_{N}$                                    &\\ \hline
         $x_{B}$ & $\left(C_{\bar{X}}+\tau I_m\right)^{-1}N$  & $\tau \left(C_{\bar{X}}+\tau I_m\right)^{-1}e_{\bar{X}}$\\ \hline
                 &                                            &
\end{tabular}
\end{equation}
\end{center}

We may apply the fractional dual algorithm for Integer Linear Programming (ILP) with Gomory cuts (see for instance \cite{PapadimitriouSteiglitz98}) for deciding
if there exists (or not) a $0-1$ solution for the system \eqref{sisreduzido}.

\begin{theorem}\cite{CardosoLuz16}
If $G$ is an adverse $\mathcal{Q}$-graph, then the fractional dual algorithm for ILP with Gomory cuts, applied to the system \eqref{sisreduzido}, yields a $0-1$
solution in a finite number of iterations.
\end{theorem}

It follows an example where a $0-1$ solution is obtained after just one simplex iteration.

\begin{example}
Let us ilustraste the application of the simplex method to the determination of a $(0,2)$-regular set in the graph $G$
depicted in Figure~\ref{star_partition}.\\
Considering as starting star set  for the eigenvalue $-2$, $X=\{a,d\}$ (see Figure~\ref{star_partition}). Then $\overline{X} = \{b, c, e, f, g\}$
and we obtain the sequence o reduced tableaux \eqref{simplex_tableau}:
\begin{center}
\begin{tabular}{r|rr|r}
              & $x_a$ &     $x_d$   &   \\ \hline
      $x_b$   &   1   &      0      & 1 \\
      $x_c$   &   0   &      1      & 1 \\
      $x_e$   &   0   &      1      & 1 \\
      $x_f$   &   1   &      0      & 1 \\
      $x_g$   &  -1   &\framebox{-1}&-1 \\ \hline
              &       &             &
\end{tabular}
\; \text{ and } \; \begin{tabular}{r|rr|r}
                     (1)    & $x_a$ & $x_g$ &   \\ \hline
                    $x_b$   &   1   & 0   & 1 \\
                    $x_c$   &  -1   &-1   & 0 \\
                    $x_e$   &  -1   &-1   & 0 \\
                    $x_f$   &   1   & 0   & 1 \\
                    $x_d$   &   1   &-1   & 1 \\ \hline
                            &       &     &
                                 \end{tabular}
\end{center}
or, alternatively,
\begin{center}
\begin{tabular}{r|rr|r}
              &   $x_a$     & $x_d$  &   \\ \hline
      $x_b$   &      1      &   0    & 1 \\
      $x_c$   &      0      &   1    & 1 \\
      $x_e$   &      0      &   1    & 1 \\
      $x_f$   &      1      &   0    & 1 \\
      $x_g$   &\framebox{-1}&  -1    &-1 \\ \hline
              &     &     &
\end{tabular} \;\text{ and } \; \begin{tabular}{r|rr|r}
                                      (2)    & $x_g$ & $x_d$ &   \\ \hline
                                     $x_b$   &  -1   &  -1   & 0 \\
                                     $x_c$   &   0   &   1   & 1 \\
                                     $x_e$   &   0   &   1   & 1 \\
                                     $x_f$   &  -1   &  -1   & 0 \\
                                     $x_a$   &  -1   &   1   & 1 \\ \hline
                                             &       &       &
                                  \end{tabular}
\end{center}
In the first case the vertex subset $\{b, d, f\}$ is obtained and in the second case we obtain $\{a, c, e\}$.
Both vertex subsets are $(0,2)$-regular.
\end{example}
%

\section{Conclusions}
The concept of convex-$QP$ graph allows us to look at graphs as being partitioned into two subsets. The subset of graphs $G$ for which
the stability number is easily determined by solving the convex quadratic program \eqref{Luz_UpperBound} associated to $G$ and the subset
formed by the remaining graphs. However, in general, it is hard to decide whether the upper bound $\upsilon(G)$ coincides with the stability
number $\alpha(G)$ and throughout more than two decades, this has been the main challenge within this topic.\\
This survey starts in Section~\ref{MS_Extensions} by relating the convex quadratic program \eqref{Luz_UpperBound} with the Motzkin-Straus
quadratic model for the determination of the clique (stability) number of a graph $G$ \cite{MotzkinStraus65} by the introduction of the
parametric quadratic programs \eqref{quadratic_parametric_family} and \eqref{Q_quadratic_family} whose optimal values, $\upsilon_G(\tau)$
and $\nu_G(\tau)$, are the inverse of each other and equal to $\alpha(G)$ in the former case ($1/\alpha(G)$ in the later) when $\tau=1$ (see
\eqref{relations_with_MS}). Notice that \eqref{Luz_UpperBound} is obtained from \eqref{quadratic_parametric_family} by setting $\tau=-\lambda_{min}(G)$.
In the remaining sections, the main published results about the characterization and properties of $\mathcal{Q}$-graphs as well as  their
recognition are presented. The focus of section~\ref{Convex-QP} is the characterization of $\mathcal{Q}$-graphs
and the study of their properties. Section~\ref{Recognition} analyzes two algorithmic strategies for the recognition of $\mathcal{Q}$-graphs.
The main obstacle for deciding wether a graph $G$ is a $\mathcal{Q}$-graph or not is the presence of an adverse subgraph $H$ of $G$. However,
if we are able to conclude that $H$ is a $\mathcal{Q}$-graph, then it follows that $G$ is also a $\mathcal{Q}$-graph. The Conjecture~\ref{conjecture}
which remains open states that every adverse graph is a $\mathcal{Q}$-graph. The recognition of an adverse $\mathcal{Q}$-graph $H$ is equivalent
to the determination of a $(0,\tau)$-regular set, with $\tau=-\lambda_{min}(H)$, which can be done by finding a $0-1$ solution for the linear
system \eqref{k_t_system} or by the application of the star sets/star complements theory using a simplex-like approach. As noted in \cite[Th. 7]{CardosoLuz16},
a graph $G$ with at least one edge is a $\mathcal{Q}$-graph if and only if there is a star set $X \subset V(G)$ associated to $\lambda_{min}(G)$
such that $\upsilon(G-X) = \upsilon(G)$.  Furthermore, there is a maximum stable set $S$ such that $S \subseteq V(G) \setminus X$.

\section*{Acknowledgement}
The author thanks the referees for their careful reading, valuable suggestions and correction of several typos in a previous version of this article.

\subsection{Funding}

This research was partially supported by the Portuguese Foundation for Science and Technology (\textquotedblleft FCT-Funda\c c\~ao para a Ci\^encia e a Tecnologia\textquotedblright),
through the CIDMA - Center for Research and Development in Mathematics and Applications, within project UID/MAT/04106/2013.


\begin{thebibliography}{12}

\bibitem{BarbosaCardoso04}
Barbosa R, Cardoso DM. On regular-stable graphs. Ars Combinatoria 2004;70:149--159.

\bibitem{Bomze}
Bomze IM. On standard quadratic optimization problems. J. Global Optim. 1998;13:369--387.

\bibitem{BrouwerHaemers12}
Brouwer AE, Haemers  WH. Spectra of Graphs. New York: Springer; 2012.

\bibitem{Cardoso01}
Cardoso DM. Convex Quadratic Programming Approach to the Maximum Matching Problem. J. Global Optim. 2001;21:91--106.

\bibitem{Cardoso03}
Cardoso DM. On graphs with stability number equal to the optimal value of a convex quadratic program.
Matem\'atica Contempor\^{a}nea 2003;21:9--24.

\bibitem{CardosoCvetkovic06}
Cardoso DM, Cvetkovi\'c D. Graphs with least eigenvalue $-2$ attaining a convex quadratic upper bound for the stability number. Bulletin T. CXXXIII
de l'Acad\'emie Serbe des Sciences et des Arts, Classe des sciences math\'ematiques et naturelles, sciences math\'ematiques 2006;31:41--55.

\bibitem{CardosoLozin12}
Cardoso DM, Lozin VV. On hereditary properties of the classe of graphs with convex quadratic stability number. J. Math. Science 2012;182(2):227--232.

\bibitem{CardosoLozinLuzPacheco16}
Cardoso DM, Lozin VV, Luz CJ, Pacheco MF. Efficient domination through eigenvalues. Discrete Applied Mathematics 2016;214:54--62.

\bibitem{CardosoLuz01}
Cardoso DM, Luz CJ. Extensions of the Motzkin-Straus Result on the Stability Number of Graphs. Cadernos de Matem\'{a}tica,
Departamento de Matem\'{a}tica da Universidade de Aveiro 2001;CM01/I-17:18 p.

\bibitem{CardosoLuz16}
Cardoso DM, Luz CJ. A simplex like approach based on star sets for recognizing convex-QP adverse graphs. J. Combinatorial Optim.
2016;31(1):311--326.

\bibitem{CardosoRama04}
Cardoso DM, Rama P. Equitable bipartitions of graphs and related results. J. Math. Science 2004;120(1):869--880.

\bibitem{CardosoScirihaZerafa10}
Cardoso DM, Sciriha I, Zerafa C. Main eigenvalues and $(k,\tau)$-regular sets. Linear Algebra Appl. 2010;423: 2399-2408.

\bibitem{CaviqueLuz09}
Cavique L, Luz CJ. A Heuristic for the Stability Number of a Graph based on Convex Quadratic Programming and Tabu Search.
J. Math. Science 2009;161(2):944--955.


\bibitem{CvetkovicRowlinson04}
Cvetkovi{\'c} D, Rowlinson P. Spectral Graph Theory in Beinek LD and Wilson RJ (Ed.), Topics in Algebraic Graph Theory.
Encyclopeddia of Mathematics and Its Applications 102. Cambridge: Cambridge University Press;2004. p. 88--112.

\bibitem{CvetRowSim2010}
Cvetkovi{\'c} D, Rowlinson P, Simi{\'c} S. An Introduction to the Theory of Graph Spectra. Cambridge: Cambridge University Press;2010.

\bibitem{Doob82}
Doob M. A Suprising Property of the Least Eigenvalue of a Graph. Linear Algebra Appl. 1982;46:1--7.

\bibitem{Edmonds65}
Edmonds JR. Paths trees and flowers.  Canad. J. Math. 1965;17:449--467.

\bibitem{Gibbons_et_al97}
Gibbons LE, Hearn DW, Pardalos PM, Ramana MV. Continuos characterisations of the maximum clique problem. Mathematics of Operations Research
1997;22:754--768.

\bibitem{Karp72}
Karp RM. Reducibility among combinatorial problems. In: Complexity of Computer Computations, eds. R.E. Miller and J. W. Thatcher,
New York: Plenum Press;1972:85--104.

\bibitem{LasVergnas75}
Las Vergnas M. A note on matchings in graphs. Colloque sur la Théorie des Graphes (Paris, 1974).
Cahiers Centre Études Recherche Opérationnelle 1975;17: 257–260.

\bibitem{Lovasz79}
Lov\'asz L. On the Shanon capacity of a graph. IEEE Transactions on Information Theory 1979;25(2):1--7.

\bibitem{Lovasz86}
Lov\'{a}sz L. An Algorithm Theory of Numbers, Graphs and Convexity. Regional Conference Series in Applied Mathematics, Philadelphia: SIAM 1986.

\bibitem{Lozin00}
Lozin VV. Stability in $P_{5}$ and Banner-free graphs.  European J. Oper. Research 2000;125:292--297.

\bibitem{Luz95}
Luz CJ. An upper bound on the independence number of a graph computable in polynomial time. Operations Research Letters 1995;18:139--145.

\bibitem{Luz95-2}
Luz CJ. Improving an upper bound on the stability number of a graph. J. Global Optim. 2005;31:61--84.

\bibitem{LuzCardoso1}
Luz CJ, Cardoso DM. A Generalization of the Hoffman-Lov\'{a}sz upper bound on the independence number of a regular graph.
Annals of Operations Research 1998;81:307--319.

\bibitem{LuzCardoso2}
Luz CJ, Cardoso DM. A quadratic programming approach to the determination of an upper bound on the weighted stability number.
 European J. Oper. Research 2001;132(3):91--103.

\bibitem{LuzSchrijver05}
Luz CJ, Schrijver A. A convex quadratic characterization of the Lov\'asz theta number. SIAM J. Discrete Math. 2005;19(2):382--387.

\bibitem{Luz16}
Luz CJ. A characterization of the weighted Lov\'{a}sz number based on convex quadratic programming. Optimization Letters 2016;10(1): 19-31.

\bibitem{Minty80}
Minty GJ. On maximal independent sets of vertices in claw-free graphs. Journal of Combinatory Theory B. 1980;28:284--304.

\bibitem{MotzkinStraus65}
Motzkin TS, Straus EG. Maxima for graphs and a new proof of a theorem of Tur{\'a}n. Can. J. Math. 1965;17:533--540.

\bibitem{PapadimitriouSteiglitz98}
Papadimitriou CH, Steiglitz K. Combinatorial Optimization: Algorithms and Complexity. Dover Publications; 1998.


\bibitem{PeliloJagota95}
Pelilo M, Jagota A. Feasible and infeasible maxima in quadratic program for maximum clique. J. Artif. Neuro Networks 1995;2:411--420.

\bibitem{Pelilo95}
Pelilo M, Relaxation labeling networks for the maximum clique problem. J. Artif. Neural Networks 1995;2:313--327.

\bibitem{Sbihi80}
Sbihi N. Algorithm de recherche d'un stable de cardinalit\'{e} maximum dans un graphe sans \'{e}toile. Discrete Mathematics. 1980;29:53--76.

\bibitem{Thompson81}
Thompson DM. Eigengraphs: constructing strongly regular graphs with block desings. Utilitas Math. 1981;20:83--115.

\bibitem{Turan41}
Tur\'an P. On an extremal problem in graph theory (in Hungarian). Matematikai \'es Fizikai Lapok  1941;48:436-452.



\end{thebibliography}
\end{document}